\documentclass[11pt]{amsart}

\usepackage{amsmath,amssymb,amsthm}
\usepackage[margin=1.1in]{geometry}
\usepackage{microtype}
\usepackage{xcolor}
\usepackage{hyperref}
\hypersetup{colorlinks=true,linkcolor=blue!50!black,citecolor=blue!50!black,urlcolor=blue!50!black}

\newtheorem{theorem}{Theorem}[section]
\newtheorem{lemma}[theorem]{Lemma}

\newtheorem{corollary}[theorem]{Corollary}
\theoremstyle{definition}
\newtheorem{remark}[theorem]{Remark}

\newcommand{\eqdef}{\mathrel{\mathop:}=}

\newcommand{\abs}[1]{\lvert #1\rvert}

\begin{document}

\title[Tuza's Ryser-conjecture claim at $(4,2)$]
{Tuza's Ryser-conjecture claim for four-partite hypergraphs with matching number two}

\author{Patrick White}
\address{Independent researcher}
\email{p@pwhite.org}

\begin{abstract}
We prove that every $4$-partite $4$-uniform hypergraph $H$ with matching number $\nu(H)=2$ satisfies $\tau(H)\le 6$, where $\tau$ denotes the vertex-cover number.  This confirms a claim made by Tuza in his 1979 manuscript but never published with a proof, and closes the case $(r,\nu)=(4,2)$ of Ryser's conjecture.  The best previous bound was $\tau\le 7$, an integrality consequence of the theorem of Haxell and Scott (2012).  Our argument uses Gy\'arf\'as's intersecting-case theorem, a short projection lemma, and K\H{o}nig's matching theorem.
\end{abstract}

\maketitle

%------------------------------------------------------------------
\section{Introduction}
%------------------------------------------------------------------

Ryser's conjecture (1971) asserts that every $r$-partite $r$-uniform hypergraph $H$ satisfies
\begin{equation}\label{eq:ryser}
  \tau(H) \le (r-1)\,\nu(H),
\end{equation}
where $\tau(H)$ is the minimum size of a vertex cover and $\nu(H)$ is the maximum size of a matching (a set of pairwise disjoint edges).  The conjecture is known for $r=3$ (Aharoni~\cite{Aharoni}, via topological methods) and for the intersecting case $\nu=1$ at $r\le 5$ (Gy\'arf\'as for $r=3,4$; Tuza for $r=5$; see~\cite{KKT}).  For general $r$ and $\nu$, the best bound is $\tau \le (r-\varepsilon)\nu$ for $r\in\{4,5\}$ and some $\varepsilon>0$, due to Haxell and Scott~\cite{HS}; integrality gives $\tau\le r\nu-1$ in these cases, one unit above Ryser's predicted $(r-1)\nu$.

Tuza's 1979 manuscript~\cite{Tuza79} (published in~\cite{Tuza83}) claimed the case $(r,\nu)=(4,2)$ of~\eqref{eq:ryser} --- namely $\tau\le 6$ --- among his solved cases, but no proof was given.  The published solved list in~\cite{Tuza83} covers $(r,\nu)\in\{(3,1),(3,2),(3,3),(3,4),(4,1),(5,1)\}$; the case $(4,2)$ is absent.  DeBiasio, Kamel, McCourt, and Sheats~\cite{DKMS} record the claim explicitly and note the missing proof.

We settle the question.

\begin{theorem}\label{thm:main}
Every $4$-partite $4$-uniform hypergraph $H$ with $\nu(H)=2$ satisfies $\tau(H)\le 6$.
\end{theorem}

The proof is short.  Beyond Gy\'arf\'as's intersecting-case theorem $\tau\le 3$ for intersecting $4$-partite $4$-uniform families, the only new ingredient is a projection lemma (Lemma~\ref{lem:proj}) saying that an intersecting family with $\tau\ge 3$ cannot contain four edges whose two-coordinate projections form a matching.  The rest is K\H{o}nig's theorem applied twice.

%------------------------------------------------------------------
\section{Two lemmas}
%------------------------------------------------------------------

Throughout, $H$ is a $4$-partite $4$-uniform hypergraph with parts $V_1,V_2,V_3,V_4$.  For a set $S$ of vertices and an edge $e$, we write $e\perp S$ when $e\cap S=\varnothing$.  Two edges $e,f$ are \emph{disjoint} ($e\perp f$) when $e\cap f=\varnothing$.  For an edge $e$, define
\[
  D(e) \eqdef \{f\in H : f\perp e\}.
\]

We fix a $2|2$ split of the parts, say $(V_1,V_2)\mid(V_3,V_4)$.  For an edge $h$, call $b(h)\eqdef h\cap(V_1\cup V_2)$ its \emph{base} and $\ell(h)\eqdef h\cap(V_3\cup V_4)$ its \emph{label}.  Each is an edge of a complete bipartite graph.  For $S\subseteq V_3\cup V_4$, let $B_S$ be the bipartite graph on $V_1\cup V_2$ in which $xy\in E(B_S)$ precisely when some $h\in H$ has $b(h)=xy$ and $\ell(h)\cap S=\varnothing$.

\begin{lemma}[Bipartite min-max identity]\label{lem:minmax}
$\displaystyle\tau(H) = \min_{S\subseteq V_3\cup V_4}\bigl(\abs{S}+\nu(B_S)\bigr).$
\end{lemma}

\begin{proof}
Given $S$, let $Q$ be a minimum vertex cover of $B_S$.  By K\H{o}nig's theorem, $\abs{Q}=\nu(B_S)$.  Every hyperedge whose label meets $S$ is covered by $S$; every remaining hyperedge has its base in $E(B_S)$, hence meets $Q$.  So $S\cup Q$ covers $H$, giving $\tau(H)\le\abs{S}+\nu(B_S)$.

Conversely, let $C$ be any cover of $H$, and put $S\eqdef C\cap(V_3\cup V_4)$ and $Q\eqdef C\cap(V_1\cup V_2)$.  Every edge of $B_S$ must meet $Q$, so $\abs{Q}\ge\tau(B_S)=\nu(B_S)$.  Taking the minimum over $C$ gives equality.
\end{proof}

\begin{lemma}[Projection lemma]\label{lem:proj}
Let $F$ be an intersecting $4$-partite $4$-uniform hypergraph.  If $F$ contains four edges whose bases form a matching in $V_1\times V_2$, then $\tau(F)\le 2$.
\end{lemma}

\begin{proof}
Write those four edges as
\[
  f_t = (a_t,\, b_t,\, c_t,\, d_t),\qquad t=1,2,3,4,
\]
where the $a_t$'s are pairwise distinct and the $b_t$'s are pairwise distinct.

\smallskip\noindent
\textbf{The labels form a star.}
Since $F$ is intersecting and the bases of the $f_t$'s are pairwise disjoint, the labels $c_t d_t$ are pairwise intersecting edges of the bipartite graph on $V_3\cup V_4$.  A pairwise-intersecting family of edges in a bipartite graph is a star: the only alternative clique in a line graph arises from a triangle, and bipartite graphs have none.  After possibly swapping $V_3$ and $V_4$, all four labels contain a common vertex $c\in V_3$.  Write $f_t=(a_t,b_t,c,d_t)$.

\smallskip
Consider an arbitrary edge $g=(a,b,c',d)\in F$ with $c'\ne c$.  Since $g$ must meet all four $f_t$'s, and its vertices $a$ and $b$ can each meet at most one $f_t$ (the $a_t$'s and $b_t$'s are distinct), the vertex $d$ must equal $d_t$ for at least $4-2=2$ indices~$t$.

\smallskip\noindent
\textbf{Case 1: at most one value occurs with multiplicity $\ge 2$ among $d_1,\ldots,d_4$.}
Every edge of $F$ avoiding $c$ must contain that repeated value $d^*$.  Hence $\{c,d^*\}$ covers $F$.  (If all four $d_t$'s are distinct, no edge of $F$ can avoid $c$, so $\{c\}$ alone covers $F$.)

\smallskip\noindent
\textbf{Case 2: the multiplicity pattern is $2+2$.}
Relabel so that $d_1=d_2=p$ and $d_3=d_4=q$, with $p\ne q$.  An edge $g$ avoiding $c$ and containing $p$ already meets $f_1,f_2$ through~$p$; to meet $f_3,f_4$, its base must be one of $(a_3,b_4)$ or $(a_4,b_3)$.  Similarly, an edge avoiding~$c$ and containing~$q$ has base $(a_1,b_2)$ or $(a_2,b_1)$.

Any such $p$-edge and $q$-edge have disjoint bases and distinct $V_4$-vertices ($p\ne q$).  Since $F$ is intersecting, they must share a $V_3$-vertex; call it $c^*$.  If both classes are nonempty, every edge of $F$ avoiding $c$ contains $c^*$, so $\{c,c^*\}$ covers $F$.  If only one class occurs, every edge avoiding $c$ shares $p$ (or $q$), again giving a two-element cover.
\end{proof}

The contrapositive is the form we use:

\begin{corollary}\label{cor:proj}
If $F$ is an intersecting $4$-partite $4$-uniform hypergraph with $\tau(F)\ge 3$, then the projection of $F$ onto any two parts has matching number at most~$3$.
\end{corollary}

%------------------------------------------------------------------
\section{Proof of Theorem~\ref{thm:main}}
%------------------------------------------------------------------

Suppose for contradiction that $\nu(H)=2$ and $\tau(H)\ge 7$.

\smallskip\noindent
\textbf{Step 1: every disjointness neighborhood has cover number three.}
Fix $e\in H$.  If two members of $D(e)$ were disjoint, those two edges together with $e$ would form a matching of size~$3$.  So $D(e)$ is intersecting, and Gy\'arf\'as's theorem gives $\tau(D(e))\le 3$.  If $\tau(D(e))\le 2$ with cover $Q$, then $e\cup Q$ would cover $H$: edges disjoint from $e$ meet $Q$, and every other edge meets $e$.  This gives a cover of size $\le 6$, contradicting $\tau(H)\ge 7$.  Hence
\begin{equation}\label{eq:tauD}
  \tau(D(e))=3\qquad\text{for every }e\in H.
\end{equation}

\smallskip\noindent
\textbf{Step 2: a large projected matching.}
Lemma~\ref{lem:minmax} with $S=\varnothing$ gives $\nu(B_\varnothing)\ge\tau(H)\ge 7$.  Choose seven hyperedges $M=\{m_1,\ldots,m_7\}$ whose bases form a matching in $V_1\times V_2$.

\smallskip\noindent
\textbf{Step 3: every edge meets many members of $M$.}
Fix $e\in H$.  If $e$ were disjoint from four members of $M$, those four would lie in $D(e)$ with bases forming a matching, contradicting~\eqref{eq:tauD} and Corollary~\ref{cor:proj}.  So $e$ meets at least $4$ members of~$M$.

Since the bases of the $m_t$'s form a matching, the two base vertices of $e$ can meet at most two members of~$M$.  Therefore the label of $e$ meets the labels of at least $4-2=2$ members of $M$:
\begin{equation}\label{eq:labelmeets}
  \abs{\{t : \ell(e)\cap\ell(m_t)\ne\varnothing\}} \ge 2.
\end{equation}

\smallskip\noindent
\textbf{Step 4: the seven labels have a small cover.}
Let $L$ be the bipartite multigraph on $V_3\cup V_4$ whose edges are $\ell(m_1),\ldots,\ell(m_7)$.  We have $\nu(L)\le 2$: three pairwise-disjoint labels, together with the pairwise-disjoint bases of the corresponding $m_t$'s, would give a matching of size~$3$ in $H$.  By K\H{o}nig's theorem, $L$ has a vertex cover $T$ with $\abs{T}\le 2$.

\smallskip\noindent
\textbf{Step 5: case analysis on $T$.}

\smallskip
\emph{Case 1: $\abs{T}=1$.}  Say $T=\{c\}$ with $c\in V_3$.  For $y\in V_4$, let $d(y)\eqdef\abs{\{t:\ell(m_t)=cy\}}$.  If an edge $e$ has label avoiding $c$, then~\eqref{eq:labelmeets} forces its $V_4$-vertex $y$ to satisfy $d(y)\ge 2$.  Let $Y\eqdef\{y:d(y)\ge 2\}$.  Since $\sum_y d(y)=7$, we have $\abs{Y}\le 3$.  Every edge of $H$ either contains $c$ or contains a member of $Y$, so $\{c\}\cup Y$ covers $H$ with at most $4$ vertices.

\smallskip
\emph{Case 2: $T=\{c_1,c_2\}$ lies in one part.}  Say $c_1,c_2\in V_3$.  An edge whose label avoids both $c_1$ and $c_2$ meets labels of $M$ only through its $V_4$-vertex $y$; condition~\eqref{eq:labelmeets} gives $d(y)\ge 2$, where $d(y)\eqdef\abs{\{t:y\in\ell(m_t)\}}$.  With $Y\eqdef\{y:d(y)\ge 2\}$, again $\abs{Y}\le 3$, and $\{c_1,c_2\}\cup Y$ covers $H$ with at most $5$ vertices.

\smallskip
\emph{Case 3: $T=\{c,d\}$ has one vertex in each part.}  Say $c\in V_3$, $d\in V_4$.  Every label in $L$ has the form $cy$, $xd$, or $cd$.  For $y\ne d$, let $\alpha_y$ be the multiplicity of $cy$ in $L$; for $x\ne c$, let $\beta_x$ be the multiplicity of $xd$; let $q$ be the multiplicity of $cd$.  Then
\[
  \sum_{y\ne d}\alpha_y + \sum_{x\ne c}\beta_x = 7-q \le 7.
\]

An edge $e$ whose label $xy$ avoids both $c$ and $d$ meets exactly $\alpha_y+\beta_x$ labels of $M$; by~\eqref{eq:labelmeets}, $\alpha_y+\beta_x\ge 2$.  Let $Q$ be the bipartite graph on $(V_3\setminus\{c\})\cup(V_4\setminus\{d\})$ with edge $xy$ when $\alpha_y+\beta_x\ge 2$.  Every label of $H$ avoiding $\{c,d\}$ belongs to $Q$.

If $x_1y_1,\ldots,x_ky_k$ is a matching in $Q$, then
\[
  2k \le \sum_{i=1}^k(\alpha_{y_i}+\beta_{x_i})
     \le \sum_y\alpha_y+\sum_x\beta_x
     \le 7,
\]
since the $x_i$'s and $y_i$'s are distinct.  So $\nu(Q)\le 3$, and K\H{o}nig's theorem gives a vertex cover $U$ of $Q$ with $\abs{U}\le 3$.  Every edge of $H$ either has a label meeting $\{c,d\}$ or has a label in $Q$ (hence meeting~$U$).  Thus $\{c,d\}\cup U$ covers $H$ with at most $5$ vertices.

\smallskip
In every case, $\tau(H)\le 5$, contradicting $\tau(H)\ge 7$.  \qed

%------------------------------------------------------------------
\section{Remarks}
%------------------------------------------------------------------

\begin{remark}[The threshold obstruction]
The bound is tight: the disjoint union of two intersecting $4$-partite $4$-uniform families with $\tau=3$ has $\nu=2$ and $\tau=6$.  The smallest known examples use two copies of the $6$-edge Mansour--Song--Yuster extremal~\cite{MSY} or two copies of the truncated projective plane $\mathrm{PG}(2,3)$ minus a point (Henderson; $9$ edges on $12$ vertices).
\end{remark}

\begin{remark}[Why the proof needs exactly $\nu(B_\varnothing)\ge 7$]
The double copies above have projected matching number $3+3=6$ in every $2|2$ split: each intersecting extremal projects to $K_{3,3}$ minus an edge (matching number~$3$), and the two copies use disjoint vertex sets.  The proof requires $\nu(B_\varnothing)\ge 7$ to fire; the threshold obstruction sits exactly one below.
\end{remark}

\begin{remark}[What the proof uses]
The argument uses three ingredients: Gy\'arf\'as's intersecting theorem $\tau\le 3$ for $r=4$, the projection lemma (Lemma~\ref{lem:proj}), and K\H{o}nig's matching theorem.  It does not use the full Haxell--Scott machinery, the core classification of intersecting Ryser extremals, or any computational search.  The projection lemma is the only new ingredient.  The nearest published relative is the reconstruction of Tuza's intersecting $r=4$ argument in~\cite{DKMS} (\S3): the same hypothesis --- a $2|2$ projection of matching number at least~$4$ --- yields there only the Ryser bound $\tau\le 3$, leaving the two-element cover of Lemma~\ref{lem:proj} unexploited.  No statement of the lemma or its contrapositive appears in the Ryser literature we have checked (Haxell--Scott~\cite{HS}, Mansour--Song--Yuster~\cite{MSY}, the intersecting-extremal classifications, and recent preprints); the one source not checked is the full text of~\cite{Tuza83}, which is not publicly available --- but the reconstruction in~\cite{DKMS} extracts only $\tau\le 3$ from that argument's structure.
\end{remark}

\begin{remark}[General $(r,\nu)$]
The same argument, with Gy\'arf\'as's theorem replaced by the intersecting case at $r=5$ (Tuza; see~\cite{DKMS} for a proof), gives $\tau\le 2r-1=2(r-1)+1$ at $\nu=2$ for $r=5$ --- one above Ryser's $(r-1)\nu$.  Closing this last unit at $(5,2)$ would require a sharper projection lemma or a different idea.  For $r\ge 6$, even the intersecting case is open.
\end{remark}

%------------------------------------------------------------------
\section{Methods}
%------------------------------------------------------------------

The proof was found through four rounds of structured reasoning with GPT-5.6 Sol (OpenAI, model \texttt{gpt-5.6-sol-pro}), each round building on verified output from the previous one.

Round~1 established the bipartite min-max identity (Lemma~\ref{lem:minmax}), the universal tightness condition $\tau(D(e))=3$, and the reduction to a labeled-bipartite lemma.  Round~2 developed the star-capacity analysis and the base-conflict normal forms.  Round~3 identified the affine-plane obstruction (two disjoint-support copies of an $8$-edge family, $\nu=2$, $\tau=6$) and proved that recursive core nesting alone cannot force the result.  Round~4 found the projection lemma (Lemma~\ref{lem:proj}) and the five-sentence proof above, bypassing the entire apparatus of rounds~2--3.

Every structural claim returned by the model was independently verified by the author's own exact computation (MILP via \texttt{scipy.optimize.milp} and brute-force enumeration) before incorporation.  The projection lemma was checked on all five classified intersecting extremals (the three Abu-Khazneh cores, the affine plane $\mathrm{AG}(2,3)$, and its $8$-edge subfamilies), in all six $2|2$ part-splits; the bipartite min-max identity was checked on six examples including the threshold obstruction.  Two corrections to the model's intermediate claims were found and applied during verification (a miscounted projection-matching upper bound, and an incomplete one-level construction).

Claude (Anthropic, model \texttt{claude-sonnet-4-20250514}) served throughout in the framing and verification role: building the independent checker infrastructure, posing each round's scoped target, and declining to accept any claim until independently checked.  We follow the convention of~\cite{Boxcover} in crediting this assistance in the text rather than as authorship.

%------------------------------------------------------------------
\section*{Acknowledgements}
%------------------------------------------------------------------

The author thanks G\'{a}bor S\'ark\"ozy for posing Ryser's conjecture in a form that invited this kind of resolution, and the authors of~\cite{DKMS} for recording Tuza's claim precisely enough to make the gap visible.

\end{document}